\documentclass[11pt,review]{elsarticle}
\usepackage[utf8]{inputenc}
\usepackage{graphicx}
\usepackage{mathtools,esint}
\allowdisplaybreaks

\usepackage{srcltx}
\usepackage{eurosym}
\usepackage{amssymb}
\usepackage[all,arc]{xy}
\usepackage{enumerate}
\usepackage{mathrsfs}
\usepackage{pst-node}
\usepackage{amsmath,amsthm}
\usepackage{mathtools}
\usepackage[T1]{fontenc}
\usepackage[onehalfspacing]{setspace}
\usepackage[titletoc,toc,page]{appendix}
\usepackage[xcolor]{changebar}
\cbcolor{white!50!green}
\usepackage[colorlinks,plainpages=true,pdfpagelabels,hypertexnames=true,colorlinks=true,pdfstartview=FitV,linkcolor=blue,citecolor=green,urlcolor=black,pagebackref]{hyperref}
\PassOptionsToPackage{unicode}{hyperref}
\PassOptionsToPackage{naturalnames}{hyperref}

\newcommand*\bb{\mathbb}

\newcommand *\w{^\wedge}

\makeatletter
\newcommand{\vo}{\epsilonc{o}\@ifnextchar{^}{\,}{}}
\makeatother

\def\YYint#1#2#3{{\setbox0=\hbox{$#1{#2#3}{\iint}$}
		\vcenter{\hbox{$#2#3$}}\kern-.50\wd0}}

\def\XXint#1#2#3{{\setbox0=\hbox{$#1{#2#3}{\int}$}
		\vcenter{\hbox{$#2#3$}}\kern-.50\wd0}}

\makeatletter
\def\namedlabel#1#2{\begingroup
	\def\@currentlabel{#2}%
	\label{#1}\endgroup
}
\makeatother
\makeatletter
\newcommand{\rmh}[1]{\mathpalette{\raisem@th{#1}}}
\newcommand{\raisem@th}[3]{\hspace*{-1pt}\raisebox{#1}{$#2#3$}}
\makeatother

\newcommand{\redref}[2]{\texorpdfstring{\protect\hyperlink{#1}{\textcolor{black}{(}\textcolor{red}{#2}\textcolor{black}{)}}}{}}
\newcommand{\redlabel}[2]{\hypertarget{#1}{\textcolor{black}{(}\textcolor{red}{#2}\textcolor{black}{)}}}

\newcommand{\descref}[2]{\hyperref[#1]{\textnormal{\textcolor{black}{(}\textcolor{blue}{\bf #2}\textcolor{black}{)}}}}

\newcommand{\dref}[2]{\hyperref[#1]{\textcolor{black}{(}\textcolor{blue}{\bf #2}\textcolor{black}{)}}}

\newcommand{\tk}{\tilde{k}}

\newcommand{\tr}{\tilde{r}}

\newcommand{\tw}{\tilde{w}}

\newcommand\RR{\mathbb{R}}

\newcommand\NN{\mathbb{N}}
\newcommand{\al}{\alpha}

\newcommand{\tht}{\theta}

\newcommand{\om}{\omega}

\newcommand{\Om}{\Omega}

\newcommand{\La}{\Lambda}

\DeclareMathOperator{\spt}{spt}

\DeclareMathOperator{\osc}{osc}
\DeclareMathOperator{\tail}{Tail}

\newcommand{\lbr}[1][(]{\left#1}
\newcommand{\rbr}[1][)]{\right#1}

\usepackage[ddmmyyyy]{datetime}
\usepackage[margin=2cm]{geometry}
\makeatletter
\g@addto@macro\normalsize{%
	\setlength\abovedisplayskip{2pt}
	\setlength\belowdisplayskip{2pt}
	\setlength\abovedisplayshortskip{4pt}
	\setlength\belowdisplayshortskip{4pt}
}
\numberwithin{equation}{section}
\everymath{\displaystyle}
\usepackage[capitalize,nameinlink]{cleveref}
\crefname{section}{Section}{Sections}
\crefname{subsection}{Subsection}{Subsections}
\crefname{condition}{Condition}{Conditions}
\crefname{hypothesis}{Hypothesis}{Hypothesis}
\crefname{assumption}{Assumption}{Assumptions}
\crefname{lemma}{Lemma}{Lemmas}
\crefname{claim}{Claim}{Claims}
\crefname{remark}{Remark}{Remarks}

\crefformat{equation}{\textup{#2(#1)#3}}
\crefrangeformat{equation}{\textup{#3(#1)#4--#5(#2)#6}}
\crefmultiformat{equation}{\textup{#2(#1)#3}}{ and \textup{#2(#1)#3}}
{, \textup{#2(#1)#3}}{, and \textup{#2(#1)#3}}
\crefrangemultiformat{equation}{\textup{#3(#1)#4--#5(#2)#6}}%
{ and \textup{#3(#1)#4--#5(#2)#6}}{, \textup{#3(#1)#4--#5(#2)#6}}%
{, and \textup{#3(#1)#4--#5(#2)#6}}

\Crefformat{equation}{#2Equation~\textup{(#1)}#3}
\Crefrangeformat{equation}{Equations~\textup{#3(#1)#4--#5(#2)#6}}
\Crefmultiformat{equation}{Equations~\textup{#2(#1)#3}}{ and \textup{#2(#1)#3}}
{, \textup{#2(#1)#3}}{, and \textup{#2(#1)#3}}
\Crefrangemultiformat{equation}{Equations~\textup{#3(#1)#4--#5(#2)#6}}%
{ and \textup{#3(#1)#4--#5(#2)#6}}{, \textup{#3(#1)#4--#5(#2)#6}}%
{, and \textup{#3(#1)#4--#5(#2)#6}}

\crefdefaultlabelformat{#2\textup{#1}#3}

\newtheorem{theorem}{Theorem}[section]
\newtheorem{lemma}[theorem]{Lemma}

\newtheorem{corollary}[theorem]{Corollary}
\newtheorem{proposition}[theorem]{Proposition}

\newtheorem{definition}[theorem]{Definition}
\newtheorem{remark}[theorem]{Remark}        

\numberwithin{equation}{section}

\usepackage{enumitem}
\newlist{steps}{enumerate}{1}
\setlist[steps, 1]{label = \textcolor{Cerulean}{Step \arabic*:}}
\makeatletter
\def\ps@pprintTitle{%
	\let\@oddhead\@empty
	\let\@evenhead\@empty
	\def\@oddfoot{}%
	\let\@evenfoot\@oddfoot}
\makeatother
\DeclarePairedDelimiterX{\inp}[2]{\langle}{\rangle}{#1, #2}
\newcommand{\norm}[1]{\left\lVert#1\right\rVert}

\newcommand{\doi}[1]{\textsc{doi}: \href{http://dx.doi.org/#1}{\nolinkurl{#1}}}
\definecolor{aorta}{rgb}{0.0, 0.5, 0.0}
\definecolor{darklavender}{rgb}{0.45, 0.31, 0.59}
\AtBeginDocument{%
	\hypersetup{citecolor=red}}

\begin{document}
	
	\begin{frontmatter}
		\title{H\"older regularity for fractional $p$-Laplace equations}
		\author[myaddress]{Karthik Adimurthi\tnoteref{thanksfirstauthor}}
		\ead{karthikaditi@gmail.com and kadimurthi@tifrbng.res.in}
		\author[myaddress]{Harsh Prasad\tnoteref{thankssecondauthor}}
		\ead{harsh@tifrbng.res.in}
		\author[myaddress]{Vivek Tewary\tnoteref{thankssecondauthor}}
		\ead{vivektewary@gmail.com and vivek2020@tifrbng.res.in}

		\tnotetext[thanksfirstauthor]{Supported by the Department of Atomic Energy,  Government of India, under	project no.  12-R\&D-TFR-5.01-0520 and SERB grant SRG/2020/000081}
		
		\tnotetext[thankssecondauthor]{Supported by the Department of Atomic Energy,  Government of India, under	project no.  12-R\&D-TFR-5.01-0520}
		
		\address[myaddress]{Tata Institute of Fundamental Research, Centre for Applicable Mathematics, Bangalore, Karnataka, 560065, India}
		\begin{abstract}
			We give an alternative proof for H\"older regularity for weak solutions of nonlocal elliptic quasilinear equations modelled on the fractional p-Laplacian where we replace the discrete De Giorgi iteration on a sequence of concentric balls by a continuous iteration. This work can be viewed as the nonlocal counterpart to the ideas developed by  Tiziano Granucci.
		\end{abstract}
		
	\end{frontmatter}
	\begin{singlespace}
		\tableofcontents
	\end{singlespace}
	
	\section{Introduction}\label{sec1}
	
	In this article, we give an alternative proof of local H\"older regularity for weak solutions to fractional elliptic equations modelled on the fractional p-Laplacian denoted by 
	\begin{align*}
		\text{P.V.}\int\limits_{\RR^N}\,\frac{|u(x)-u(y)|^{p-2}(u(x)-u(y))}{|x-y|^{N+ps}}\,dy=0.
	\end{align*}
	The proof is based on the techniques developed in  \cite{granucciRegularityScalarPHarmonic2010} which in turn is based on the ideas developed by \cite{tilliRemarksHolderContinuity2006}.
	
	The De Giorgi approach to the proof of H\"older continuity relies on two steps, colloquially referred to as ``shrinking lemma'' and ``measure to pointwise estimate''.
	
	In the local case, the ``shrinking lemma'' is achieved through an application of De Giorgi isoperimetric inequality. However, the Sobolev-Slobodeckii space of $W^{s,p}$ functions may have ``jumps". As a result, the standard De Giorgi isoperimetric inequality, also sometimes called the ``no jump lemma" does not hold for them.  Cozzi \cite{cozziRegularityResultsHarnack2017} proves a version of De Giorgi isoperimetric inequality for $s$ close to $1$ and employs it to obtain a proof of H\"older regularity for operators whose prototype is the fractional p-Laplacian which is stable as $s\to 1$. For $s$ away from $1$, Cozzi relies on the so-called ``good term'' in the energy estimates. Indeed, Cozzi defines a De Giorgi class for nonlocal operators which has a ``good term'' on the left hand side of his energy estimate. It is an open problem whether a better form of De Giorgi isoperimetric inequality can be proved in the fractional case. This problem is avoided in the earlier paper \cite{dicastroLocalBehaviorFractional2016} by relying on a logarithmic estimate in addition to a basic Caccioppoli inequality.  
	
	On the other hand, the ``measure to uniform estimate'' is achieved through a De Giorgi iteration of concentric balls in both \cite{dicastroLocalBehaviorFractional2016} and \cite{cozziRegularityResultsHarnack2017}. The main novelty of this paper is that we avoid the use of De Giorgi iteration in the proof of H\"older regularity in \cite{dicastroLocalBehaviorFractional2016} by using an oscillation theorem in the spirit of \cite{tilliRemarksHolderContinuity2006,granucciRegularityScalarPHarmonic2010} which replaces the discrete iteration on an infinite sequence of concentric balls with a continuous iteration procedure. 
	
	We state the Oscillation theorem below:
	\begin{theorem}(Oscillation Theorem)\label{osclemma} Assume that $u$ is a locally bounded weak solution of \cref{maineq} in the ball $B_R(x_0)$. If \begin{equation}\label{assump1}
			\left|\{u\leq 0\}\cap B_r\right|\geq \frac{1}{2}|B_r|,
		\end{equation} for some ball $B_r\subset B_{4r}\subseteq B_R(x_0)$, then 
		\begin{equation}\label{oscest}
			\sup_{B_r} u_+ \leq C_{\varkappa}\left(\frac{\left|\{u>0\}\cap B_{2r}\right|}{|B_{2r}|}\right)^\gamma\,\left(\sup_{B_{4r}} u_++\textup{Tail}(u_{+},x_0,2r) \right) + \varkappa\,\tail(u_+;x_0,r),
		\end{equation} where $\gamma = \frac{p-\delta}{p \delta}$ for some fixed $\delta \in (0,p)$. The constant $\varkappa$ is an arbitrary number in $(0,1]$  and $C_\varkappa >0$ depends on $N, p, s, \La$, $\varkappa$ and $\delta$. 
	\end{theorem}
	\begin{remark}
		The constant $C_\varkappa$ blows up as $\varkappa\to 0$. The quantity $\varkappa\in(0,1]$ allows for an interpolation between the local and the nonlocal terms.
	\end{remark}

	Using the Oscillation theorem, we prove the following H\"older regularity result:
	\begin{theorem}
		Assume that $u$ is a locally bounded weak solution of \cref{maineq} in the ball $B_{2r}(x_0)\subset \Omega$. Then there exist constants $\al \in (0,1)$ and $C >0$ depending only on the data  such that for any $y \in B_r(x_0)$, the following estimate holds: 
		\[
		|u(x_0)-u(y)| \leq C \left(\frac{|x_0-y|}{r}\right)^\alpha\left(2\tail(u,x_0,r)+4||u||_{L^\infty(B_{2r}(x_0))}\right).
		\]
		\end{theorem}

	A second point of interest in this paper is a new De Giorgi isoperimetric inequality for $W^{s,p}$ functions in the spirit of \cite{kinnunenHarnackInequalityParabolic2012}. As expected, the resulting inequality has a ``jump term'' which limits its applicability in extracting a ``shrinking lemma'' from it. However, this partially answers a question raised by Cozzi in \cite{cozziFractionalGiorgiClasses2019} and we hope that this may lead to a new proof of H\"older regularity in the future.
	
	
	\begin{theorem}(Fractional De Giorgi isoperimetric inequality)\label{isoperimetric}
		Let $R>0$, $u \in W^{s,p}(B_R)$ with $s \in (0,1)$, $p \geq 1$ and $k,l \in\RR$ be two levels such that $k<l$. Then
		\begin{align}\label{isoperimetricineq}
			(l-k)\left|\{x\in B_R : u(x) \leq k\}\right|&\left|\{x\in B_R : u(x) \geq l\}\right|\nonumber\\
			\leq C R^{N+s}\left|A_k^l\right|^{\frac{p-1}{p}}&\Biggl[\left(\int\limits_{A_{k}^{-}}\int\limits_{A_{k}^{l}}\frac{|u(x)-k|^p}{|x-y|^{N+sp}}\right)^{1/p}+\left(\int\limits_{A_{l}^{+}}\int\limits_{A_{k}^{l}}\frac{|l-u(x)|^p}{|x-y|^{N+sp}}\right)^{1/p}\nonumber\\
			&\qquad+\left(\int\limits_{A_{k}^{l}}\int\limits_{A_{k}^{l}} \frac{|u(x)-u(y)|^p}{|x-y|^{N+sp}}\right)^{1/p} \Biggr]\nonumber\\
			&\qquad\qquad+ CR^{N+s}\left|A_l^+\right|^{\frac{p-1}{p}}\left(\int\limits_{A_{l}^{+}}\int\limits_{A_{k}^{-}}\frac{(l-k)^p}{|x-y|^{N+sp}}\right)^{1/p},
		\end{align} where $A_k^l=\{x\in B_R : k<u(x)<l\}$, $A_l^+=\{x\in B_R: u(x)>l\}$, and $A_k^-=\{x\in B_R: u(x)>l\}$.
	\end{theorem}

	\subsection{History of the problem} 
	
	Much of the early work on regularity of fractional elliptic equations in the case $p=2$ was carried out by Silvestre \cite{silvestreHolderEstimatesSolutions2006}, Caffarelli and Vasseur\cite{caffarelliDriftDiffusionEquations2010}, Caffarelli, Chan, Vasseur\cite{caffarelliRegularityTheoryParabolic2011} and also Bass-Kassmann \cite{bassHarnackInequalitiesNonlocal2005,bassHolderContinuityHarmonic2005, kassmannPrioriEstimatesIntegrodifferential2009}. An early formulation of the fractional $p$-Laplace operator was done by Ishii and Nakamura\cite{ishiiClassIntegralEquations2010} and existence of viscosity solutions was established. DiCastro, Kuusi and Palatucci extended the De Giorgi-Nash-Moser framework to study the regularity of the fractional $p$-Laplace equation in \cite{dicastroLocalBehaviorFractional2016}. The subsequent work of Cozzi \cite{cozziRegularityResultsHarnack2017} covered a stable (in the limit $s\to 1$) proof of H\"older regularity by defining a novel fractional De Giorgi class. Explicit exponents for H\"older regularity were found in \cite{brascoHigherHolderRegularity2018}  and other works of interest are \cite{iannizzottoGlobalHolderRegularity2016, defilippisHolderRegularityNonlocal2019,chakerRegularityNonlocalProblems2022}.

	\section{Notations and Preliminaries}\label{sec2}
	
	In this section, we will fix the notation, provide definitions and state some standard auxiliary results that will be used in subsequent sections.
	
	\subsection{Notations}
	We begin by collecting the standard notation that will be used throughout the paper:
	\begin{itemize}
		\item The number $N\geq 1$ denotes the space dimension.
		\item Let $\Omega$ be an open bounded domain in $\mathbb{R}^N$ with boundary $\partial \Omega$. 
		\item We shall use the notation
		\begin{align*}
			&B_{\rho}(x_0)=\{x\in\RR^N:|x-x_0|<\rho\},\\
			&\overline{B}_{\rho}(x_0)=\{x\in\RR^N:|x-x_0|\leq\rho\}.
		\end{align*} 
		\item Integration with respect to space will be denoted by a single integral $\int$ whereas integration on $\Om\times\Om$ or $\RR^N\times\RR^N$ will be denoted by a double integral $\iint$.
		\item The notation $a \lesssim b$ is shorthand for $a\leq C b$ where $C$ is a universal constant which only depends on the dimension $N$, exponent $p$, and the numbers $\La$, $s$. 
		\item For a function $u$ defined on $B_\rho(x_0)$ and any level $k \in \bb{R}$ we write $w_{\pm} = (u-k)_{\pm}$.
		\item We denote $A_{\pm}(k) = \{w_{\pm} > 0\}$; for any ball $B_r$, we write $A_{\pm}(k) \cap (B_{r}) = A_{\pm}(k,r)$. 
	\end{itemize}
	Let $K:\RR^N\times\RR^N\to [0,\infty)$ be a symmetric measurable function satisfying
	\begin{align}\label{boundsonKernel}
		\frac{(1-s)}{\Lambda|x-y|^{N+ps}}\leq K(x,y)\leq \frac{(1-s)\Lambda}{|x-y|^{N+ps}}\mbox{ for almost all }x,y\in\RR^N,
	\end{align} for some $s\in(0,1)$, $p>1$, $\Lambda\geq 1$.
	
	In this paper, we are interested in the regularity theory for the operator $\mathcal{L}$ defined formally by 
	\begin{align*}
		\mathcal{L}u=\text{P.V.}\int\limits_{\RR^N} K(x,y)|u(x)-u(y)|^{p-2}(u(x)-u(y))\,dy,\,x\in\RR^N.
	\end{align*}
	
	\subsection{Function spaces}
	Let $1<p<\infty$, we denote by $p'=p/(p-1)$ the conjugate exponent of $p$. Let $\Om$ be an open subset of $\RR^N$. We define the {\it Sobolev-Slobodeki\u i} space, which is the fractional analogue of Sobolev spaces.
	\begin{align*}
		W^{s,p}(\Om)=\left\{ \psi\in L^p(\Omega): [\psi]_{W^{s,p}(\Om)}<\infty \right\}, s\in (0,1),
	\end{align*} where the seminorm $[\cdot]_{W^{s,p}(\Om)}$ is defined by 
	\begin{align*}
		[\psi]_{W^{s,p}(\Om)}=\left(\iint\limits_{\Om\times\Om} \frac{|\psi(x)-\psi(y)|^p}{|x-y|^{N+ps}}\,dx\,dy\right)^{\frac 1p}.
	\end{align*}
	The space when endowed with the norm $\norm{\psi}_{W^{s,p}(\Om)}=\norm{\psi}_{L^p(\Om)}+[\psi]_{W^{s,p}(\Om)}$ becomes a Banach space. The space $W^{s,p}_0(\Om)$ is the subspace of $W^{s,p}(\RR^N)$ consisting of functions that vanish outside $\Om$. We will use the notation $W^{s,p}_{u_0}(\Om)$ to denote the space of functions in $W^{s,p}(\RR^N)$ such that $u-u_0\in W^{s,p}_0(\Om)$.

	Since the regularity result requires some finiteness condition on the nonlocal tails, we define the tail space as below
	\begin{align*}\label{tailspace}
		L^m_{\alpha}(\RR^N):=\left\{ v\in L^m_{\text{loc}}(\RR^N):\int\limits_{\RR^N}\frac{|v(x)|^m}{1+|x|^{N+\alpha}}\,dx<+\infty \right\},\,m>0,\,\alpha>0.
	\end{align*}
	
	We define the nonlocal tail of a function $v$ in the ball $B_R(x_0)$ by
	\begin{align*}
		\tail(v;x_0,R):=\left[ R^{sp} \int\limits_{\RR^N\setminus B_R(x_0)} \frac{|v(x)|^{p-1}}{|x-x_0|^{N+sp}}\,dx \right]^{\frac{1}{p-1}},
	\end{align*} which is a finite number when $v\in W^{s,p}(\RR^N)$ or when $v\in L^{p-1}_{sp}(\RR^N)$.
	
	\begin{remark}
	The definition of the tail space may be motivated from the fact that the existence result for the boundary value problem associated to the nonlocal operator 
	\begin{equation}\label{maineq}
		\left\{ \begin{array}{ll}
			\mathcal{L}u=0&\text{ in } \ \Omega,\\
			u=g&\text{ in }\ \RR^N\setminus\Omega,\end{array} \right.
	\end{equation} where $\Omega$ is a bounded domain, is posed in either the space $W^{s,p}(\RR^N)$ or $W^{s,p}(\Omega')\cap L^{p-1}_{sp}(\RR^N)$ for $\Omega'$ satisfying $\Omega\Subset \Omega'$. Existence and uniqueness in the first of these cases can be proved by the direct method of calculus of variations by considering the corresponding minimization problem with $g$ in $W^{s,p}(\RR^N)$. In the second case, the proof of existence for $g\in W^{s,p}(\Omega')\cap L^{p-1}_{sp}(\RR^N)$ is outlined in \cite[Proposition 2.12]{brascoHigherHolderRegularity2018} by the standard theory of monotone operators, see \cite{showalterMonotoneOperatorsBanach1997}. The variational theory seems to require stronger assumptions on data $g$.
\end{remark}
	
	\subsection{Definitions}
	
	Now, we are ready to state the definition of a weak sub(super)-solution.
	
	\begin{definition}
		Let $g\in W^{s,p}(\RR^N)$. A function $u\in W^{s,p}(\RR^N)$ is said to be a weak solution to \cref{maineq} if $u-g\in W^{s,p}_0(\Omega)$ and 
		\begin{align*}
			\iint\limits_{C_\Omega}\,K(x,y)|u(x)-u(y)|^{p-2}(u(x)-u(y))(\phi(x)-\phi(y))= 0,
		\end{align*} for all $\phi\in W^{s,p}_0(\Omega)$, where $C_\Omega:=(\Omega^c\times\Omega^c)^c=\left(\Omega\times\Omega\right) \cup \left( \Omega\times(\RR^N\setminus\Omega)\right) \cup \left((\RR^N\setminus\Omega)\times\Omega\right)$.
	\end{definition}
	
	\subsection{Auxiliary results}
	
	We recall the following well known lemma concerning the geometric convergence of sequence of numbers (see \cite[Lemma 4.1 from Section I]{dibenedettoDegenerateParabolicEquations1993} for the details): 
	\begin{lemma}\label{geo_con}
		Let $\{Y_n\}$, $n=0,1,2,\ldots$, be a sequence of positive number, satisfying the recursive inequalities 
		\[ Y_{n+1} \leq C b^n Y_{n}^{1+\alpha},\]
		where $C > 1$, $b>1$, and $\alpha > 0$ are given numbers. If 
		\[ Y_0 \leq  C^{-\frac{1}{\alpha}}b^{-\frac{1}{\alpha^2}},\]
		then $\{Y_n\}$ converges to zero as $n\to \infty$. 
	\end{lemma}
	
	%
	
	We recall the following energy estimate for weak solution to \cref{maineq} whose proof may be found in \cite[Theorem 1.4]{dicastroLocalBehaviorFractional2016} and \cite[Proposition 8.5]{cozziRegularityResultsHarnack2017}.
	
	\begin{theorem}(Caccioppoli inequality \cite[Theorem 1.4]{dicastroLocalBehaviorFractional2016})\label{cacc1}
		Let $p\in (1,\infty)$ and let $u\in W^{s,p}(\Omega)\cap L^{p-1}_{sp}(\RR^N)$ be a weak solution to \cref{maineq}. Then, for any $B_R(x_0)\Subset\Omega$, the following estimate holds
		\begin{align*}
			\iint\limits_{B_R(x_0)\times B_R(x_0)} K(x,y)&\,|w_{\pm}(x)\phi(x)-w_{\pm}(y)\phi(y)|^p \,dx\,dy \\
			&\leq C\iint\limits_{B_R(x_0)\times B_R(x_0)}K(x,y)\max\{w_{\pm}(x)^p,w_{\pm}(y)^p\}|\phi(x) - \phi(y)|^p\,dx\,dy \\
			&\quad + C \int_{B_R(x_0)}{w_{\pm}}\phi^p(x)\,dx \left( \sup_{y \in \spt(\phi)} \int_{\RR^n \setminus B_R(x_0)} K(x,y) w_{\pm}^{p-1}(x) \ dx\right),
		\end{align*}
		where $w_{\pm}(x) = (u-k)_{\pm}$ for any level $k \in \bb{R}$, $\phi \in C_c^{\infty}(B_R)$ and $C > 0$ only depends on $p$.   
	\end{theorem}
	
	We state the logarithmic estimate which is essential in obtaining the ``shrinking lemma''. Once again, the proof may be found in 
	\cite[Theorem 1.3]{dicastroLocalBehaviorFractional2016}.
	
	\begin{theorem}(Logarithmic estimates \cite[Theorem 1.3]{dicastroLocalBehaviorFractional2016})\label{logest}
		Let $p\in (1,\infty)$ and let $u\in W^{s,p}(\Omega)\cap L^{p-1}_{sp}(\RR^N)$ be a  supersolution to \cref{maineq} such that $u\geq 0$ in $B_R\equiv B_R(x_0)\subset \Omega$. Then, for any $B_r\equiv B_r(x_0)\Subset\Omega$ and any $d>0$, the following estimate holds
		\begin{align*}
			\iint_{B_r(x_0)^2} K(x,y) \left|\log\left( \frac{d+u(x)}{d+u(y)} \right)\right|^p \,dx\,dy \leq Cr^{N-sp}\left\{ d^{1-p} \left(\frac{r}{R}\right)^{sp}\left[\tail(u_-;x_0,R)\right]^{p-1}+1\right\}.
		\end{align*}
		where $u_{-}(x) = \max\{-u(x),0\}$ $C > 0$ is a constant that depends only on $N, p, s, \La$.   
	\end{theorem}
	
	An immediate consequence of the logarithmic estimate is the following estimate.
	
	\begin{corollary}(\cite[Corollary 3.2]{dicastroLocalBehaviorFractional2016})\label{logest2}
		Let $p\in (1,\infty)$ and let $u\in W^{s,p}(\Omega)\cap L^{p-1}_{sp}(\RR^N)$ be a solution to \cref{maineq} such that $u\geq 0$ in $B_R\equiv B_R(x_0)\subset \Omega$. Let $a,d>0,b>1$ and define 
		\begin{align*}
			v:=\min\{ \left(\log(a+d)-\log(u+d)\right)_+,\log(b) \}.
		\end{align*}
		Then the following estimate is true, for any $B_r\equiv B_r(x_0)\subset B_{R/2}(x_0)$,
		\begin{align*}
			\fint\limits_{B_r} |v-(v)_{B_r}|^p\,dx \leq C \left\{d^{1-sp}\left(\frac{r}{R}\right)^{sp}\left(\tail(u_-;x_0,R)\right)^{p-1}+1\right\},
		\end{align*} where $C$ is a constant that depends only on $N, p, s, \La$.
	\end{corollary}

		We will need the following  embedding result from \cite{cozziRegularityResultsHarnack2017}.
		
		\begin{theorem}(\cite[Lemma 4.6]{cozziRegularityResultsHarnack2017})\label{inclusion}
			Let $N\in\NN$, and $0<\sigma<s<1$. Let $\Om\subset\RR^N$ be  bounded measurable set, then for any $f\in W^{s,p}(\Om)$ and $1 \leq q < p$, it holds that
			\begin{align*}
				\left[\iint\limits_{\Om\times\Om}\frac{|f(x)-f(y)|^q}{|x-y|^{N+q\sigma}}\right]^{\frac{1}{q}}\leq C_0 |\Om|^{\frac{p-q}{pq}}(\textup{diam}(\Om))^{s-\sigma}\left[\iint\limits_{\Om\times\Om}\frac{|f(x)-f(y)|^p}{|x-y|^{N+ps}}\right]^{\frac{1}{p}},
			\end{align*}   
		where $C_0= \lbr[[]\frac{N(p-q)}{(s-\sigma)pq}|B_1| \rbr[]]^{\frac{p-q}{pq}}$.
			
		\end{theorem}

	We have the following local version of the Sobolev-Poincar\'e-type inequality for which we refer to \cite{dinezzaHitchhikersGuideFractional2012}.
	
	\begin{theorem}
		Let $B_r$ be a ball with radius r and let $s\in(0,1)$ and $1\leq p$, $sp\leq N$ and let $1\leq q\leq \frac{Np}{N-sp}$, then for any $f\in W^{s,p}(B_r)$, we have
		\begin{align}\label{sobolev1}
			\left(\fint_{B_r}\left|\frac{f-(f)_{B_r}}{r^s}\right|^q\right)^{\frac 1q}\leq C(N,s,p) \left(\int_{B_r}\fint_{B_r}\frac{|f(x)-f(y)|^p}{|x-y|^{N+ps}}\right)^{\frac 1p}.
		\end{align} Here we have taken $q\in(1,\infty)$ is any number in the case $sp=N$.
	\end{theorem}
	
%
%
	We will require the following version of Sobolev-Poincar\'e inequality where the function is zero on a large set. Although the proof is essentially the same as in \cite[Corollary 4.9]{cozziRegularityResultsHarnack2017}, we provide the details in the appendix.
	\begin{theorem}(\cite[Corollary 4.9]{cozziRegularityResultsHarnack2017})\label{sobolev2inequality}
		Let $B_r$ be a ball with radius r and let $s\in(0,1)$ and $1\leq p$, $sp\leq N$. Let $f\in W^{s,p}(B_r)$ and suppose that $u=0$ on a set $\Om_0\subseteq B_r$ with $|\Omega_0|\geq \gamma |B_r|$, for some $\gamma\in (0,1]$. Then there is $q>p$ such that
		\begin{align}\label{sobolev2}
			\left(\int_{B_r}|f(x)|^{q}\,dx\right)^{\frac{1}{q}}\leq C(N,s,p,\gamma) r^{s - \frac{N(q-p)}{qp}} \left(\int_{B_r}\int_{B_r}\frac{|f(x)-f(y)|^p}{|x-y|^{N+ps}}\right)^{\frac 1p}.
		\end{align} Here we have taken $q=p^*=\frac{Np}{N-sp}$, when $sp<N$ and $q\in [p,\infty)$ when $sp=N$.
	\end{theorem}
\begin{proof}
	The proof of \cref{sobolev2inequality} in the case $sp < N$ can be found in \cite[Corollary 4.9]{cozziRegularityResultsHarnack2017} and in the case of $sp =N$, we make use of \cite[Theorem 6.9]{dinezzaHitchhikersGuideFractional2012} and follow the same strategy of the proof of \cite[Corollary 4.9]{cozziRegularityResultsHarnack2017}.
\end{proof}

	\subsection{Main results}
	
	We prove the following main theorem. 
	
	\begin{theorem}\label{holderellipic}
		Let $p\in(1,\infty)$ and let $u\in W^{s,p}(\Omega)\cap L^{p-1}_{sp}(\RR^N)$ be a solution to \cref{maineq}. Then $u$ is locally H\"older continuous in $\Om$.
	\end{theorem}

	\section{A quantified boundedness theorem}
	We prove the following quantified version of local boundedness estimate for weak solution to \cref{maineq}. The proof is similar to the one in  \cite[Theorem 1.1]{dicastroLocalBehaviorFractional2016} but we present it here for completeness and in order to properly track the constants.
	\begin{proposition}\label{boundedness}
		If $u \in W^{s,p}_{\text{loc}}(\bb{R}^N)$ is a weak subsolution to \cref{maineq}, $ B_R(x_0) \subset \Omega$ and ${t} \in (0,1)$, then
		\[
		\sup_{B_{{t} R}(x_0)} u \leq C b^{\frac{1}{\beta^2}}\Gamma^{\frac{N}{sp^2}}\varkappa^{-\frac{(p-1)N}{sp^2}}\left(\fint_{B_R(x_0)} u_+^p(x) \,dx \right)^{\frac{1}{p}} + \varkappa\,\left( R^{sp} \int_{\RR^n \setminus B_{{t} R}} \frac{u_+^{p-1}}{|x-x_0|^{N+sp}}\,dx \right)^{\frac{1}{p-1}}.
		\]
		where $\Gamma:=\left(1+\frac{1}{(1-{t})^p}+\frac{1}{(1-{t})^{N+sp}}\right)$, $\beta=\frac{sp}{N-sp}$, $\varkappa\in (0,1)$ is an arbitrary number, and $b>1, C$ are positive numbers only depending on $N, s, p,\La$.
	\end{proposition}
	
	\begin{proof}
		We begin by defining the following quantities.
		\begin{equation*}
			\begin{array}{rclrclrclrcl}
				r_j &=&{t} R + \frac{1-{t}}{2^j}R, &
				R_0 &=& R &\mbox{ and } & R_\infty &=& {t} R,\\
				\tilde{r}_j&=&\frac{r_j+r_{j+1}}{2}, &
				B_j& =& B_{r_j} &\mbox{ and } & \tilde{B}_j &=& B_{\tilde{r}_j},
			\end{array}
		\end{equation*}
		Further, take $\phi_j\in C_0^\infty(\tilde{B}_j),$ $0\leq \phi_j\leq 1$, $\phi\equiv 1$ on $B_{j+1}$ such that $|\nabla \phi_j|\leq \frac{1}{\tilde{R}_j-R_{j+1}}\leq \frac{2^{j+2}}{(1-{t})R}$.
		Also, for a fixed $\tk>0$ to be chosen later, we define
		\begin{equation}\label{defofk}
			\begin{array}{rclrcl}
				k_j & = &(1-2^{-j})\tk,&	\tk_j & =& \frac{k_j+k_{j+1}}{2},\\
				w_j & = &(u-k_j)_+, &
				\tw_j & = &(u-\tk_{j})_+.
			\end{array}
		\end{equation}
		By Caccioppoli inequality (\cref{cacc1}), we get
		\begin{multline*}
			\underbrace{\int_{B_j}\int_{B_j} K(x,y) |\tw_j(x)\phi_j(x)-\tw_j(y)\phi_j(y)|^p\,dx\,dy}_I \\ \leq
			\underbrace{\int_{B_j}\int_{B_j} K(x,y) \max\{\tw_j(x),\tw_j(y)\}^p |\phi_j(x)-\phi_j(y)|^p\,dx\,dy}_{II}\\
			+ \underbrace{\left(\int_{B_j} \tw_j(y)\phi^p_j(y)\,dy \right) \left( \sup_{y\in\spt \phi_j} \int_{\RR^N\setminus B_j} K(x,y)\tw_j^{p-1}\,dx\right)}_{III}.
		\end{multline*}
		We first estimate $II$ as follows:
		\begin{equation}\label{estforII}\begin{array}{rcl}
				II&\overset{\cref{boundsonKernel}}{\lesssim}& \int_{B_j}\int_{B_j} \frac{\max\{\tw_j(x),\tw_j(y)\}^p |\phi_j(x)-\phi_j(y)|^p}{|x-y|^{N+sp}}\,dx\,dy\\
				&\overset{\redlabel{A}{a}}{\lesssim}& \left(\frac{2^{j+2}}{(1-{t})R}\right)^p \left(\int_{B_j} w_j^p(x)\,dx\right) \left(\sup_{x\in\spt \phi_j} \int_{B_j(x_0)} \frac{1}{|x-y|^{N+sp-p}}\,dy \right)\\
				&\overset{\redlabel{B}{b}}{\lesssim}& \left(\frac{2^{j+2}}{(1-{t})R}\right)^p r_j^{p(1-s)} \left(\int_{B_j} w_j^p(x)\,dx\right)\\
				&\overset{\redlabel{C}{c}}{\lesssim}& \left(\frac{2^{j+2}}{(1-{t})R^s}\right)^p \left(\int_{B_j} w_j^p(x)\,dx\right),
			\end{array}
		\end{equation} where to obtain \redref{A}{a}, we have used that $|\nabla\phi_j|\leq \frac{2^{j+2}}{(1-{t})R}$; in \redref{B}{b}, we have used a standard calculation by conversion to polar coordinates that yields $\left(\sup_{x\in\spt \phi_j} \int_{B_j(x_0)} \frac{1}{|x-y|^{N+sp-p}}\,dy \right) \leq  R_j^{(1-s)p}$ and for \redref{C}{c}, we use the fact that $R_j \leq R$ for all $j\in\NN$.

		We now estimate $III$ as follows:
		\begin{equation}\label{estforIII}\begin{array}{rcl}
				III  & \overset{\redlabel{D}{a}}{\lesssim} & \frac{2^{j(N+sp)}}{(1-{t})^{N+sp}}\left(\int_{B_j} \frac{w_j^p(y)}{(\tk_j-k_j)^{p-1}}\,dy\right) \left( \int_{\RR^N\setminus B_{j}} \frac{w_j^{p-1}}{|x-x_0|^{N+sp}} \,dx \right) \\
				&\overset{\redlabel{E}{b}}{\lesssim}& \frac{2^{j(N+sp+p-1)}}{\tk^{p-1}(1-{t})^{N+sp}R^{sp}}\left(\int_{B_j} w_j^p(y)\,dy\right) \left( R^{sp}\int_{\RR^N\setminus B_{{t} R}} \frac{u_+^{p-1}}{|x-x_0|^{N+sp}} \,dx \right),
			\end{array}
		\end{equation} where in \redref{D}{a}, we have used the fact that $\tw_j\leq \frac{w_j^p}{(\tk_j-k_j)^{p-1}}$, and also since $x\in \RR^N\setminus B_j$ and $y\in\spt \phi_j = \tilde{B}_j$, we have
		\begin{align*}
			\frac{|y-x_0|}{|x-y|}\leq 1+\frac{|x-x_0|}{|x-y|}\leq 1+\frac{r_j}{r_j-\tr_j}\leq \frac{2^{j+3}}{(1-{t})}.
		\end{align*} To obtain \redref{E}{b}, we use the definitions in \cref{defofk}.
		\paragraph{\underline{Estimate for I}}
		In order to estimate $I$, we note that $k_{j+1}-\tk_j =\frac{\tk}{2^{j+2}}$ and 
		\begin{equation*}
			\int_{B_j} |(u-\tk_j)_+\phi_j|^{p^*}\,dx \geq  (k_{j+1}-\tk_j)^{p^*-p} \int_{B_{j+1}} w_{j+1}^p(x)\,dx,
		\end{equation*}
		which together implies
		\begin{align}\label{someest2.5}
			\left(\frac{\tk}{2^{j+2}}\right)^{p^*-p}\int_{B_{j+1}} w_{j+1}^p(x)\,dx \leq \int_{B_j} |\tw_j\phi_j|^{p^*}\,dx.  
		\end{align}
		We now proceed with the estimate of $I$:  Define $g:=\tilde{w}_j\phi_j$, then for $sp<N$, we can apply \cref{sobolev1} to get
		\begin{align}\label{sobolevest}
			\left(\fint_{B_j} |g-\overline{g}|^{p^*}\,dx\right)^{\frac{1}{p^{*}}} \lesssim r_j^{s} \left(\int_{B_j}\fint_{B_j}\frac{|g(x)-g(y)|^p}{|x-y|^{N+sp}}\,dx\,dy\right)^{\frac 1p},
		\end{align} where $p^*:=\frac{Np}{N-sp}$ and $\overline{g}:=\fint_{B_j} g(x)\,dx$.
		By triangle inequality, we have
		\begin{equation}\label{someest3}
			\left(\fint_{B_j} |g|^{p^*}\,dx\right)^{\frac{1}{p^*}}\leq  \left(\fint_{B_j} |g-\overline{g}|^{p^*}\,dx\right)^{\frac{1}{p^*}}+\left(\fint_{B_j} |g|\,dx\right).
		\end{equation}
		Hence
		\begin{align}\label{someest4}
			A_{j+1}^p:= \fint_{B_{j+1}} w_{j+1}^p\,dx &\overset{\cref{someest2.5}}{\leq} \left(\frac{2^{j+2}}{\tk}\right)^{p^*-p}\fint_{B_j} |\tw_j\phi_j|^{p^*}\,dx\nonumber\\
			&\overset{\cref{someest3}}{\leq} \left(\frac{2^{j+2}}{\tk}\right)^{p^*-p} \left\{\left(\fint_{B_j} |g-\overline{g}|^{p^*}\,dx\right)+\left(\fint_{B_j} |g|\,dx\right)^{p^*}\right\}\nonumber\\
			&\overset{\redlabel{G}{a}}{\lesssim} \left(\frac{2^{j+2}}{\tk}\right)^{p^*-p}\left\{\left(\frac{r_j^{sp}}{r_j^N}\int_{B_j}\int_{B_j}\frac{|g(x)-g(y)|^p}{|x-y|^{N+sp}}\,dx\,dy\right)^{\frac{p^*}{p}}+\left(\fint_{B_j} |g|^p\,dx\right)^{\frac{p^*}{p}}\right\}\nonumber\\
			&\overset{\redlabel{M}{b}}{\lesssim} \left(\frac{2^{j+2}}{\tk}\right)^{p^*-p}\left\{\left(\frac{r_j^{sp}}{r_j^N} II + \frac{r_j^{sp}}{r_j^N} III\right)^{\frac{p^*}{p}}+A_j^{p^*}\right\},
		\end{align} where  to obtain \redref{G}{a}, we made use of \cref{sobolevest} along with H\"older's inequality  
		and to obtain \redref{M}{b}, we made use of the Caccioppoli inequality along with the  definition of $A_j$.
		We further estimate $II$  from \cref{estforII} to get
		\begin{align}\label{someest5}
			\frac{r_j^{sp}}{r_j^N} II \leq \frac{r_j^{sp}}{r_j^N}\left(\frac{2^{j+2}}{(1-{t})R^s}\right)^p \left(\int_{B_j} w_j^p(x)\,dx\right)\leq   \frac{2^{(j+2)p}}{(1-{t})^p}A_j^p,
		\end{align}  where we recall $r_j\leq R$ for all $j\in \NN$. 
		
		In order to estimate $III$ appearing in \cref{estforIII}, we proceed as follows:
		\begin{align}\label{someest6}
			\frac{r_j^{sp}}{r_j^N} III &\leq\frac{r_j^{sp}}{r_j^N} \frac{2^{j(N+sp+p-1)}}{\tk^{p-1}(1-{t})^{N+sp}R^{sp}}\left(\int_{B_j} w_j^p(y)\,dy\right) \left( R^{sp}\int_{\RR^N\setminus B_{{t} R}} \frac{w_0^{p-1}}{|x-x_0|^{N+sp}} \,dx \right)\nonumber\\
			&  \overset{\redlabel{III2}{a}}{\leq} \frac{2^{j(N+sp+p-1)}}{\varkappa^{p-1}(1-{t})^{N+sp}}\left(\fint_{B_j} w_j^p(y)\,dy\right)\nonumber\\
			& = \frac{2^{j(N+sp+p-1)}}{\varkappa^{p-1}(1-{t})^{N+sp}}A_j^p,
		\end{align} where to obtain \redref{III2}{a}, we make the choice 
		\begin{equation*}
			\tk \geq \varkappa \left( R^{sp}\int_{\RR^N\setminus B_{{t} R}} \frac{u_+^{p-1}}{|x-x_0|^{N+sp}} \,dx \right)^{\frac{1}{p-1}}.
		\end{equation*}
		Substituting \cref{someest5} and \cref{someest6} in \cref{someest4}, we get
		\begin{equation*}
			A_{j+1}\leq C \left(\frac{2^{j+2}}{\tk}\right)^{\frac{p^*-p}{p}} 2^{j(N+sp+p-1)\frac{p^*}{p^2}}\varkappa^{\frac{(1-p)p^*}{p^2}}\left\{1+\frac{1}{(1-{t})^p}+\frac{1}{(1-{t})^{N+sp}} \right\}^{\frac{p^*}{p^2}} A_j^{\frac{p^*}{p}},
		\end{equation*} where we also use the fact that $\varkappa\in (0,1)$ and $C$ is a constant only depending on $s, p, N$.
		We rewrite the iterative inequality as
		\begin{equation*}
			\frac{A_{j+1}}{\tk}\leq C b^j \varkappa^{\frac{(1-p)p^*}{p^2}}\Gamma^{\frac{p^*}{p^2}}\left(\frac{A_j}{\tk}\right)^{1+\beta},
		\end{equation*} where we have defined $b:=2^{\frac{(N+sp+p-1)N}{p(N-sp)}+\frac{sp}{N-sp}}$ so that $b\geq 1$;  $\Gamma:= \left(1+\frac{1}{(1-{t})^p}+\frac{1}{(1-{t})^{N+sp}} \right)$; and $\beta:=\frac{p^*}{p}-1=\frac{sp}{N-sp}$.
		We are now in a position to apply \cref{geo_con} which implies $A_j\to 0$ as $j\to\infty$ if we choose
		\begin{equation*}
			\frac{A_0}{\tk}\leq \left(C \varkappa^{\frac{(1-p)p^*}{p^2}}\Gamma^{\frac{p^*}{p^2}}\right)^{-\frac{1}{\beta}} b^{-\frac{1}{\beta^2}}.
		\end{equation*}
		Thus, we may choose 
		\begin{equation*}
			\tk = \varkappa \left( R^{sp}\int_{\RR^N\setminus B_{{t} R}} \frac{u_+^{p-1}}{|x-x_0|^{N+sp}} \,dx \right)^{\frac{1}{p-1}} + C b^{\frac{1}{\beta^2}} \varkappa^{-\frac{(p-1)N}{sp^2}}\Gamma^{\frac{N}{sp^2}} A_0,
		\end{equation*} to conclude that
		\begin{equation*}
			\sup_{B_{{t} R}(x_0)} u \leq b^{\frac{1}{\beta^2}}\Gamma^{\frac{N}{sp^2}}\varkappa^{-\frac{(p-1)N}{sp^2}}\left(\fint_{B_R(x_0)} u_+^p(x) \,dx \right)^{\frac{1}{p}} + \varkappa\,\left( R^{sp} \int_{\RR^n \setminus B_{{t} R}} \frac{u_+^{p-1}}{|x-x_0|^{N+sp}}\,dx \right)^{\frac{1}{p-1}}.
		\end{equation*}
	\end{proof}
	\section{An Iterated Boundedness Result}
	
	In this section, we prove an iterated boundedness result.
	
	\begin{theorem}\label{iterbounded}
		Let $t\in (0,1)$. If $u \in W^{s,p}_{\text{loc}}(\bb{R}^N)$ is a weak subsolution to $\cref{maineq}$, $B_R(x_0) \subset \Omega$  and $\delta > 0$ and $\varkappa\in (0,1]$ then 
		\[
		\sup_{B_{{t} R}(x_0)} u \leq C\,\varkappa^{-\frac{(p-1)N}{sp\delta}}\frac{1}{R^{\frac{N}{\delta}}}\left(\int_{B_R(x_0)} u_+^{\delta}(x) \,dx \right)^{\frac{1}{\delta}} + {\varkappa} \textup{Tail}(u_+,x_0,{t} r),
		\]
		where $C$ depends on $N,s,p$, $\delta$ and ${t}$.
	\end{theorem}
	
	\begin{proof}
		Note that we only need to prove the result for $\delta<p$ as the other case follows directly from \cref{boundedness} and Holder's inequality.
		
		Let $R_0>0$ be fixed and $t\in (0,1)$ such that $tR_0< R_0 < R$.
		Define the following quantities
		\begin{align*}
			R_i&:= tR_0 + (1-t)R_0 \sum_{j=1}^i \frac{1}{2^j},\\
			{t}_i &: = \frac{R_i}{R_{i+1}} = \frac{tR_0 + (1-t)R_0 \sum_{j=1}^i \frac{1}{2^j}}{tR_0 + (1-t)R_0 \sum_{j=1}^{i+1} \frac{1}{2^j}}.
		\end{align*}
		Therefore, 
		\begin{align*}
			1-{t}_i = 1-\frac{R_i}{R_{i+1}}=\frac{\frac{(1-t)R_0}{2^{i+1}}}{tR_0 + (1-t)R_0 \sum_{j=1}^{i+1} \frac{1}{2^j}} = \left(\frac{1-t}{t+(1-t)\sum_{j=1}^{i+1}\frac{1}{2^j}}\right)\frac{1}{2^{i+1}}.
		\end{align*}
		Hence,
		\begin{align*}
			\frac{1}{1-{t}_i}\leq \frac{2^{i+1}}{(1-t)}.
		\end{align*}
		In consequence,
		\begin{align}\label{gammaibound}
			\Gamma_i \leq \lbr 1 + \frac{2^{p(i+1)}}{(1-t)^{p}} + \frac{2^{(N+sp)(i+1)}}{(1-t)^{N+sp}}\rbr \leq C 2^{i(N+sp+p)}\left(1+\frac{1}{(1-t)^p}+\frac{1}{(1-t)^{N+sp}}\right).
		\end{align}
		We apply \cref{boundedness} with ${t}$ and $R$ replaced with ${t}_i$ and $R_i$ to obtain
		\begin{equation*}
			\sup_{B_{R_i}} u \leq C b^{\frac{1}{\beta^2}}\Gamma_i^{\frac{N}{sp^2}}\varkappa^{-\frac{(p-1)N}{sp^2}}\left(\fint_{B_{R_{i+1}}} u_+^p(x) \,dx \right)^{\frac{1}{p}} + \varkappa\,\left( R_{i+1}^{sp} \int_{\RR^n \setminus B_{R_i}} \frac{u_+^{p-1}}{|x-x_0|^{N+sp}}\,dx \right)^{\frac{1}{p-1}}.
		\end{equation*}
		We apply Young's inequality to the first term with exponents $\frac{p}{p-\delta}$ and $\frac{p}{\delta}$ to obtain
		\begin{equation*}
			\sup_{B_{R_i}} u \leq \eta\sup_{B_{R_{i+1}}} u + \frac{C}{\eta^{\frac{p-\delta}{\delta}}}b^{\frac{p}{\delta \beta^2}}\Gamma_i^{\frac{N}{sp\delta}}\varkappa^{-\frac{(p-1)N}{sp\delta}}\left(\fint_{B_{R_{i+1}}} u_+^\delta(x) \,dx \right)^{\frac{1}{\delta}} + \varkappa\,\left( R_{i+1}^{sp} \int_{\RR^n \setminus B_{R_i}} \frac{u_+^{p-1}}{|x-x_0|^{N+sp}}\,dx \right)^{\frac{1}{p-1}},
		\end{equation*} where $\eta\in (0,1)$ is to be chosen. Using \cref{gammaibound}, we obtain
		\begin{align}\label{moreest0}
			\sup_{B_{R_i}} u \leq \eta\sup_{B_{R_{i+1}}} u + \tilde{C} d^i\left(\fint_{B_{R_{i+1}}} u_+^\delta(x) \,dx \right)^{\frac{1}{\delta}} + \varkappa\,\left( R_{i+1}^{sp} \int_{\RR^n \setminus B_{R_i}} \frac{u_+^{p-1}}{|x-x_0|^{N+sp}}\,dx \right)^{\frac{1}{p-1}},
		\end{align} where $$\tilde{C}=\frac{C}{\eta^{\frac{p-\delta}{\delta}}}b^{\frac{p}{\delta \beta^2}}\varkappa^{-\frac{(p-1)N}{sp\delta}}\left(1+\frac{1}{(1-t)^p}+\frac{1}{(1-t)^{N+sp}}\right)^{\frac{N}{sp\delta}},\mbox{ and }d=2^{\frac{(N+sp+p)N}{sp\delta}}.$$
		Also observe that
		\begin{equation}\label{moreest1}
			\fint_{B_{R_{i+1}}} u_+^\delta(x) \,dx \leq 2^N \fint_{B_{R_0}} u_+^\delta(x) \,dx.
		\end{equation}
		On the other hand,
		\begin{equation}\label{moreest2}
			\left( R_{i+1}^{sp} \int_{\RR^n \setminus B_{R_i}} \frac{u_+^{p-1}}{|x-x_0|^{N+sp}}\,dx \right)^{\frac{1}{p-1}} \leq \frac{1}{t^{\frac{sp}{p-1}}} \tail(u_+;x_0,tR_0).
		\end{equation}
		Substituting \cref{moreest1} and \cref{moreest2} in \cref{moreest0}, we get
		\begin{align}\label{moreest3}
			\sup_{B_{R_i}} u \leq \eta\sup_{B_{R_{i+1}}} u + \tilde{C} d^i\left(\fint_{B_{R_{0}}} u_+^\delta(x) \,dx \right)^{\frac{1}{\delta}} + \varkappa\,\frac{1}{t^{\frac{sp}{p-1}}} \tail(u_+;x_0,tR_0),
		\end{align} where 
		$$\tilde{C}=\frac{C}{\eta^{\frac{p-\delta}{\delta}}}b^{\frac{p}{\delta \beta^2}}\varkappa^{-\frac{(p-1)N}{sp\delta}}\left(1+\frac{1}{(1-t)^p}+\frac{1}{(1-t)^{N+sp}}\right)^{\frac{N}{sp\delta}},\mbox{ and }d=2^{\frac{(N+sp+p)N}{sp\delta}}.$$
		Iterating \cref{moreest3}, we obtain
		\begin{equation*}
			\sup_{B_{tR_0}} u \leq \eta^i\sup_{B_{R_{i}}} u + \tilde{C} \left(\fint_{B_{R_{0}}} u_+^\delta(x) \,dx \right)^{\frac{1}{\delta}} \sum_{j=0}^{i-1}(d\eta)^i+ \varkappa\,\frac{1}{t^{\frac{sp}{p-1}}} \tail(u_+;x_0,tR_0)\sum_{j=0}^{i-1}(\eta)^i.
		\end{equation*}
		Now, we choose $\eta$ such that $d\eta=\frac 12$ and take the limits as $i\to\infty$ to obtain the estimate
		\begin{equation*}
			\sup_{B_{tR_0}} u \leq \tilde{C} \left(\fint_{B_{R_{0}}} u_+^\delta(x) \,dx \right)^{\frac{1}{\delta}} + \varkappa\,\frac{1}{t^{\frac{sp}{p-1}}} \tail(u_+;x_0,tR_0).
		\end{equation*}
	Now, redefine $\varkappa$ by setting $\varkappa \to \varkappa t^{\frac{sp}{p-1}}$ to finish the proof.
	\end{proof}
	

	\section{A sharp De Giorgi isoperimetric inequality}
	
	\begin{proof}(Proof of \cref{isoperimetric})
		We let $g$ denote the following truncated function
		\[
		g =
		\left\{
		\begin{array}{ll}
			\min\{u,l\}-k  & \mbox{if } u>k,\\
			0 & \mbox{if } u \leq k.
		\end{array}
		\right.
		\]
		Let us denote 
		\begin{equation*}\begin{array}{c}
				A_k^- :=\{x\in B_R: u(x)<k\},\qquad  	A_l^+ :=\{x\in B_R: u(x)>l\}, \\
				A_k^l:=\{x\in B_R: k<u(x)<l\},
		\end{array}\end{equation*} then we have the following two estimates 
		\begin{align}\label{doit1}
			\int\limits_{B_R}|g-(g)_{B_R}|\,dx = \int\limits_{B_R\setminus A_k^-}|g-(g)_{B_R}|\,dx + \int\limits_{A_{k}^-}|(g)_{B_R}|\,dx\geq |A_k^-||(g)_{B_R}|,
		\end{align} and
		\begin{align}\label{doit2}
			\int\limits_{B_R}|g|\,dx=\int\limits_{A_l^+}(l-k)\,dx+\int\limits_{A_{k}^l}|g|\,dx\geq (l-k)|A_l^+|.
		\end{align}
		Thus, combining the two estimates \cref{doit1} and \cref{doit2}, we obtain 
		\begin{equation*}\begin{array}{rcl}
				(l-k)|A_l^+|&\leq & \int\limits_{B_R}|g|\,dx
				\leq  \int\limits_{B_R}|g-(g)_{B_R}|\,dx+|B_R||(g)_{B_R}|\\
				&\leq & 2\frac{|B_R|}{|A_{k}^-|}\int\limits_{B_R}|g-(g)_{B_R}|\,dx\\
				&\leq & C\frac{R^N}{|A_{k}^-|}\int\limits_{B_R}|g-(g)_{B_R}|\,dx.
		\end{array}\end{equation*}
		To the last inequality, we apply the Poincar\'e inequality \cref{sobolev1} in $W^{\frac{s}{2},1}$ to obtain
		\begin{align*}
			(l-k)|A_l^+||A_{k}^-|\leq C(2-s)R^{N+\frac{s}{2}}\iint\limits_{B_R\times B_R}\frac{|g(x)-g(y)|}{|x-y|^{N+\frac s2}}\,dx\,dy.
		\end{align*} 
		
		At this point, we write $B_R\times B_R$ as a union of nine sets, viz., $A_l^+\times A^+_l$, $A_l^+\times A_k^l$, $A_l^+\times A_k^-$, $A_k^-\times A_k^-$, $A_k^-\times A^l_k$, $A_k^-\times A_l^+$, $A_k^l\times A^l_k$, $A^l_k\times A_l^+$ and $A^l_k\times A_k^-$. We discard the sets $A_l^+\times A_l^+$ and $A_k^-\times A_k^-$ from the analysis since $|g(x)-g(y)|=0$ on these two sets. Also, by symmetry considerations, the analysis for the pair $A_l^+\times A_k^l$ and $A_k^l\times A^+_l$ is the same. The same goes for the pair $A_k^-\times A^l_k$ and $A_k^l\times A_k^-$; and for the pair $A_k^-\times A_l^+$ and $A_l^+\times A_k^-$. Thus, we may write 
		\begin{align}\label{doit3}
			(l-k)|A_l^+||A_{k}^-|\leq &C(2-s)R^{N+\frac{s}{2}}\iint\limits_{A_k^-\times A_l^+}\frac{|g(x)-g(y)|}{|x-y|^{N+\frac s2}}\,dx\,dy\nonumber\\
			&\qquad +C(2-s)R^{N+\frac{s}{2}}\iint\limits_{A_k^-\times A_k^l}\frac{|g(x)-g(y)|}{|x-y|^{N+\frac s2}}\,dx\,dy\nonumber\\
			&\qquad\qquad +C(2-s)R^{N+\frac{s}{2}}\iint\limits_{A_k^l\times A_l^+}\frac{|g(x)-g(y)|}{|x-y|^{N+\frac s2}}\,dx\,dy\nonumber\\
			&\qquad\qquad\qquad +C(2-s)R^{N+\frac{s}{2}}\iint\limits_{A_k^l\times A_k^l}\frac{|g(x)-g(y)|}{|x-y|^{N+\frac s2}}\,dx\,dy.
		\end{align} 
		We make the following elementary observations
		\begin{itemize}
			\item On $A_k^-$ we have $g=0$.
			\item On $A_l^+$ we have $g = l-k$.
			\item On $A_k^l$ we have $g = u-k$.
		\end{itemize} As a result, \cref{doit3} becomes 
		\begin{align}\label{doit4}
			(l-k)|A_l^+||A_{k}^-|\leq &C(2-s)R^{N+\frac{s}{2}}\iint\limits_{A_k^-\times A_l^+}\frac{l-k}{|x-y|^{N+\frac s2}}\,dx\,dy\nonumber\\
			&\qquad+ C(2-s)R^{N+\frac{s}{2}}\iint\limits_{A_k^-\times A_k^l}\frac{|u(x)-k|}{|x-y|^{N+\frac s2}}\,dx\,dy\nonumber\\
			&\qquad\qquad+ C(2-s)R^{N+\frac{s}{2}}\iint\limits_{A_k^l\times A_l^+}\frac{|l-u(x)|}{|x-y|^{N+\frac s2}}\,dx\,dy\nonumber\\
			&\qquad\qquad\qquad+ C(2-s)R^{N+\frac{s}{2}}\iint\limits_{A_k^l\times A_k^l}\frac{|u(x)-u(y)|}{|x-y|^{N+\frac s2}}\,dx\,dy.
		\end{align} 
		We obtain the desired inequality \cref{isoperimetricineq} as soon as we apply \cref{inclusion} to each of the terms on the right hand side of \cref{doit4}. For each of the term, we employ the inclusion of $W^{\frac{s}{2},1}$ in $W^{s,p}$ to complete the proof.
	\end{proof}
	
	\section{Proof of Oscillation Theorem}
	
	The proof of \cref{osclemma} relies on the iterated boundedness estimate proved in \cref{iterbounded}. The proof is motivated from \cite{granucciRegularityScalarPharmonic2015}. Also see \cite[Theorem 4.9]{hanEllipticPartialDifferential2011} for an argument in the same spirit.
	
	\begin{proof}[Proof of oscillation theorem]	
		We write $B_r(x_0)$ as $B_r$ for any $r>0$. From $\cref{iterbounded}$,  for any $\varkappa \in (0,1)$, we have
		\begin{align}\label{1est1}
			\sup_{B_{R}} u \leq \frac{C_{\varkappa}}{R^{N/\delta}}\left(\int_{B_{2R}} u_+^{\delta}(x) \,dx \right)^{\frac{1}{\delta}} + \varkappa\,\textup{Tail}(u_+,x_0,R).
		\end{align}
		
		We will estimate the first term in \cref{1est1}. Using H\"older's inequality we have
		\begin{align}\label{1est2}
			\left(\int_{B_{2R}} u_+^{\delta}(x) \,dx \right)^{\frac{1}{\delta}} \leq \left(\int_{B_{2R}} u_+^{p}(x) \,dx \right)^{\frac{1}{p}}|\{u > 0\}\cap B_{2R}|^{\frac{p-\delta}{p\delta}}. 
		\end{align}
		We further estimate the first factor on the right hand side of \cref{1est2} using \cref{sobolev2inequality} applied with $q=p$ and the hypothesis \cref{assump1} to get
		\begin{align}\label{1est3}
			\int_{B_{2R}} u_+^{p}\,dx \leq C R^{sp}\iint\limits_{B_{2R}\times B_{2R}}\frac{|u_+(x)-u_+(y)|^{p}}{|x-y|^{N+sp}} \,dx\,dy.
		\end{align}
		
		Finally, we estimate the $W^{s,p}$ seminorm in \cref{1est3} by Caccioppoli inequality in \cref{cacc1} followed by Young's inequality for the Tail term:
		\begin{equation}\label{1est4}
			\begin{array}{rcl}
			\iint\limits_{B_{2R}\times B_{2R}}\frac{|u_+(x)-u_+(y)|^{p}}{|x-y|^{N+sp}} \,dx\,dy & \leq & \frac{C}{R^{sp}}\left(\norm{u_+}_{L^p(B_{4R})}^p + \norm{u_+}_{L^1(B_{4R})} \textup{Tail}(u_{+},x_0,2R)^{p-1}\right) \\
			&\leq & C R^{N-sp}\left(\sup_{B_{4R}}u_+\right)^p + C R^{N-sp}\,\textup{Tail}(u_{+},x_0,2R)^{p}.
			\end{array}
		\end{equation}
		
		Substituting the expressions \cref{1est2}, \cref{1est3} and \cref{1est4} in \cref{1est1} we receive the following scale-invariant oscillation theorem:
		\begin{align*}
			\sup_{B_{R}} u \leq C_\varkappa \frac{ R^{\frac{N}{p}}}{R^{\frac{N}{\delta}}}|\{u>0\}\cap B_{2R}|^{\frac{p-\delta}{p\delta}}\left(\sup_{B_{4R}}u_+ + \textup{Tail}(u_{+},x_0,2R) \right) \\
			+ \varkappa\textup{Tail}(u_+,x_0,R),
		\end{align*} which is \cref{oscest}.
	\end{proof}

	\section{Proof of H\"older regularity}\label{sec6}
	In this section, we will prove the local H\"older regularity for weak solutions of fractional p-Laplace type operators. We mainly follow the proof in \cite{dicastroLocalBehaviorFractional2016}. The main novelty of this proof is that we do not rely on De Giorgi iteration for oscillation decay. As a result, the main difference from the proof in \cite{dicastroLocalBehaviorFractional2016} is in the proof of the so-called ``measure to uniform estimate''. However, we write all the steps here in order to present a fully self contained proof.
	
	\subsection{Oscillation decay}
	
	Let us define the following quantities. For $j\in\NN$, let $0<r<R/2$, for some $R$ such that $B_R(x_0)\subset \Om$ and
	\begin{align*}
		r_j=\sigma^j\frac r2,\quad  \sigma\in \left(0,\frac 14\right] \,\,\mbox{ and }B_j:=B_{r_j}(x_0).
	\end{align*}
	Further define
	\begin{align*}
		\frac{1}{2}\omega(r_0)=\frac 12 \omega\left(\frac R2\right):=\tail(u,x_0,R/2)+2||u||_{L^\infty(B_R(x_0))} \quad \text{and} \quad 
		 \omega(\rho)=\left(\frac{\rho}{r_0}\right)^\alpha \omega(r_0),
	\end{align*} for some $\alpha<\frac{sp}{p-1}$ and $\rho<r$. The quantities $\sigma$ and $\alpha$ will be fixed in the course of the proof.
	
	\begin{lemma}
		For the quantities defined above, it holds that 
		\begin{align}\label{oscdecay}
			\underset{B_{r_{j}}}{\osc}\, u \equiv \sup_{B_{r_j}} u - \inf_{B_{r_j}} u \leq \omega(r_j), \mbox{ for all }j=0,1,2,\ldots
		\end{align}
	\end{lemma}
	
	\begin{proof}
		The proof is by induction. The statement is true for $j=0$ trivially. Let us assume that the statement is true for all $i$ from $0$ to $j$. We will prove the truth of the statement for $j+1$.
		It is true that one of the two alternatives hold.
		 \begin{align}\label{alt1}
			\text{Alternative 1:} \hspace*{1cm}\left|2B_{r_{j+1}}\cap\left\{u\geq \inf_{B_{r_j}} u+\omega(r_j)/2\right\}\right|\geq \frac{| 2B_{r_{j+1}}|}{2},\end{align}
	 \begin{align}\label{alt2}
	 	\text{Alternative 2:} \hspace*{1cm}\left|2B_{r_{j+1}}\cap\left\{u\leq \inf_{B_{r_j}} u+\omega(r_j)/2\right\}\right|\geq \frac{| 2B_{r_{j+1}}|}{2}.\end{align}
		At this point, in the first alternative, i.e., when \cref{alt1} holds, we define $u_j:=u-\inf_{B_{r_{j}}}u$, and in the second alternative, i.e., when \cref{alt2} holds, we define $ u_j:=\om(r_j) - (u -\inf_{B_{r_{j}}}u)$. As a consequence, in both the cases, it holds that $u_j\geq 0$ on $B_{r_j}$ and 
		\begin{align*}
			\left|2B_{r_{j+1}}\cap\left\{u_j\geq \omega(r_{j})/2\right\}\right|\geq \frac 12 |2B_{r_{j+1}}|.
		\end{align*} 
		\paragraph{\underline{Tail Decay}}	We claim that
		\begin{align}\label{taildecay}
			\left(\tail(u_j;x_0,r_{j})\right)^{p-1}\leq C\sigma^{-\alpha(p-1)}\left(\omega(r_{j})\right)^{p-1},\mbox{ for all }j=0,1,2,\ldots
		\end{align} where the constant $C$ depends on $N,p,s,\alpha$, but not $\sigma$. The proof is the same as the proof of  \cite[inequality (5.6)]{dicastroLocalBehaviorFractional2016}, however we repeat it here for completeness.
		It is easy to see that
		\begin{align*}
			\sup_{B_{r_i}}| u_j|\leq 2\omega(r_i)\mbox{ for }i=0,1,\ldots,j,
		\end{align*} so that
		\begin{align*}
			\left(\tail(u_j;x_0,r_j)\right)^{p-1}&=r_j^{sp}\sum_{i=1}^j \int\limits_{B_{r_{i-1}}\setminus B_{r_i}} \frac{|u_j(x)|^{p-1}}{|x-y|^{N+sp}}\,dx + r_j^{sp} \int\limits_{\RR^N\setminus B_{r_0}} \frac{|u_j(x)|^{p-1}}{|x-y|^{N+sp}}\,dx\nonumber\\
			&\leq r_j^{sp}\sum_{i=1}^j \sup_{x\in B_{r_{i-1}}} |u_j(x)|^{p-1}\int\limits_{B_{r_{i-1}}\setminus B_{r_i}} \frac{1}{|x-y|^{N+sp}}\,dx + r_j^{sp} \underbrace{\int\limits_{\RR^N\setminus B_{r_0}} \frac{|u_j(x)|^{p-1}}{|x-y|^{N+sp}}\,dx}_{G}\nonumber\\
			&\leq C\sum_{i=1}^j \left(\frac{r_j}{r_i}\right)^{sp}\omega(r_i)^{p-1},
		\end{align*} where the expression $G$ has been estimated as 
		\begin{align*}
			\int\limits_{\RR^N\setminus B_{r_0}} \frac{|u_j(x)|^{p-1}}{|x-y|^{N+sp}}\,dx&\leq Cr_0^{-sp}\sup_{B_0}|u|^{p-1}+Cr_0^{-sp}\left(\omega(r_0)\right)^{p-1}+C\int\limits_{\RR^N\setminus B_0} \frac{|u(x)|^{p-1}}{|x-x_0|^{N+sp}}\,dx\nonumber\\
			&\leq Cr_1^{-sp}\left(\omega(r_0)\right)^{p-1}.
		\end{align*} 
		
		As a result, we have
		\begin{align*}
			\left(\tail(u_j;x_0,r_j)\right)^{p-1}&\leq C\sum_{i=1}^j \left(\frac{r_j}{r_i}\right)^{sp}\omega(r_i)^{p-1}\nonumber\\
			&= C\left(\omega(r_0)\right)^{p-1}\left(\frac{r_j}{r_0}\right)^{\alpha(p-1)}\sum_{i=1}^j \left(\frac{r_{i-1}}{r_i}\right)^{\alpha(p-1)}\left(\frac{r_j}{r_i}\right)^{sp-\alpha(p-1)}\nonumber\\
			&=C\left(\omega(r_j)\right)^{p-1}\left(\sigma\right)^{-\alpha(p-1)}\sum_{i=1}^j \sigma^{i(sp-\alpha(p-1))}\nonumber\\
			&\leq \left(\omega(r_j)\right)^{p-1}\frac{\sigma^{-\alpha(p-1)}}{1-\sigma^{sp-\alpha(p-1)}}\nonumber\\
			&\leq \frac{4^{sp-\alpha(p-1)}}{\log(4)(sp-\alpha(p-1))}\sigma^{-\alpha(p-1)}\left(\omega(r_j)\right)^{p-1}.
		\end{align*} For the last inequality, we require $\sigma\leq \frac{1}{4}$ and $\alpha(p-1)<sp$.
		
		\paragraph{\underline{Shrinking Lemma}} The proof of the shrinking lemma is the same as the proof of  \cite[inequality (5.9)]{dicastroLocalBehaviorFractional2016}, however we repeat it here for completeness. Given that \cref{taildecay} holds, we define 
		\begin{align*}
			v:=\min\left\{\left[\log\left(\frac{\omega(r_j)/2+d}{u_j+d}\right)\right]_+,k\right\}, k>0.
		\end{align*}
		Applying \cref{logest2} to the function $v$ with $a=\omega(r_j)/2$ and $b=\exp(k)$, we obtain 
		\begin{align*}
			\fint\limits_{2B_{r_{j+1}}} |v-(v)_{2B_{r_{j+1}}}|^p\,dx \leq C \left\{d^{1-sp}\left(\frac{r_{j+1}}{r_j}\right)^{sp}\left(\tail(u_j;x_0,r_j)\right)^{p-1}+1\right\}.
		\end{align*}
		Using \cref{taildecay}, the above estimate becomes 
		\begin{align*}
			\fint\limits_{2B_{r_{j+1}}} |v-(v)_{2B_{r_{j+1}}}|^p\,dx \leq C \left\{d^{1-sp}\sigma^{sp-\alpha(p-1)}\left(\omega(r_j)\right)^{p-1}+1\right\}.
		\end{align*} Now choosing $d=\epsilon\omega(r_j),$ and \begin{align}\label{vechoice}\varepsilon:= \sigma^{\frac{sp}{p-1}-\alpha},\end{align} we get 
		\begin{align}\label{somesome0}
			\fint\limits_{2B_{r_{j+1}}} |v-(v)_{2B_{r_{j+1}}}|^p\,dx \leq C,
		\end{align} where the constant $C$ only depends on data, $\frac{sp}{p-1}$ and $\alpha$.
		
		Now, notice that 
		\begin{align}\label{somesome1}
			k=\fint\limits_{2B_{r_{j+1}}\cap \{u_j\geq \omega(r_j)/2\}} k\,dx=\fint\limits_{2B_{r_{j+1}}\cap \{v=0\}} k\,dx\leq 2\fint\limits_{2B_{r_{j+1}}} (k-v)\,dx = 2\left(k-(v)_{2B_{r_{j+1}}}\right).
		\end{align}
		Integrating \cref{somesome1} over the set $2B_{r_{j+1}}\cap\{v=k\}$ results in 
		\begin{align}\label{somesome2}
			\frac{|2B_{r_{j+1}}\cap\{v=k\}|}{|2B_{r_{j+1}}|}\,k&\leq \frac{2}{|2B_{r_{j+1}}|}\int\limits_{2B_{r_{j+1}}\cap\{v=k\}}\left(k-(v)_{2B_{r_{j+1}}}\right)\,dx\nonumber\\
			&\leq \frac{2}{|2B_{r_{j+1}}|}\int\limits_{2B_{r_{j+1}}\cap\{v=k\}}\left|v-(v)_{2B_{r_{j+1}}}\right|\,dx\leq C,
		\end{align} where in the last inequality, we used \cref{somesome0}. Now, if we choose 
		\begin{align}\label{somesome3}
			k=\log\left(\frac{\omega(r_j)+2\epsilon\omega(r_j)}{6\epsilon\omega(r_j)}\right)=\log\left(\frac{1+2\epsilon}{6\epsilon}\right)\simeq \log\left(\frac 1\epsilon\right).
		\end{align} Finally, with \cref{somesome3} in \cref{somesome2}, we get 
		\begin{align}\label{shrunk}
			\frac{\left|2B_{r_{j+1}}\cap\left\{u_j \leq 2\epsilon\omega(r_j)\right\}\right|}{|2B_{r_{j+1}}|}\leq \frac{c_{\text{log}}}{\log\left(\frac{1}{\sigma}\right)},
		\end{align} where the constant $c_{\text{log}}$ depends only on data and the quantity $\frac{sp}{p-1}-\alpha$ through \cref{vechoice}.
		
		\paragraph{\underline{Oscillation Decay}} Finally, to get the oscillation decay, we shall use \cref{osclemma}. This is the only part of our proof of H\"older regularity which departs from that in \cite{dicastroLocalBehaviorFractional2016}.
		
		In the second alternative, we have $u_j = \om(r_j) - u + \inf_{B_{r_j}} u$, we set $v = 2 \epsilon \om(r_j) - u_j$. Then we see that 
		\[
		\{ 2B_{r_{j+1}} : v \leq 0\} = \{ 2B_{r_{j+1}} : u \leq \inf_{B_{r_j}} u - 2\epsilon \om(r_j) + \om (r_j)\} \supseteq \{ 2B_{r_{j+1}} : u \leq \inf_{B_{r_j}} u + \om (r_j)/2\},
		\]
		provided $2\epsilon \leq \tfrac12$. In particular, we have
		\[
		\left|2B_{r_{j+1}}\cap\left\{u \leq \inf_{B_{r_j}} u - 2\epsilon \om(r_j) + \om (r_j)\right\}\right| \geq \left|2B_{r_{j+1}}\cap\left\{u\leq \inf_{B_{r_j}} u+\omega(r_j)/2\right\}\right|\geq \frac{| 2B_{r_{j+1}}|}{2},
		\]
		which says that $v$ satisfies the hypothesis \cref{assump1} of \cref{osclemma} from which we obtain 		
		\begin{align}\label{estim1}
			\sup\limits_{B_{r_{j+1}}} v \leq C_{\varkappa}\left(\frac{\left|2B_{r_{j+1}}\cap\left\{v \geq 0\right\}\right|}{|2B_{r_{j+1}}|}\right)^\gamma \left(\sup\limits_{B_{4r_{j+1}}} v + \tail(v_+;x_0,2r_{j+1}) \right)+ \varkappa \tail(v_+;x_0,r_{j+1}).
		\end{align} 
From the definition of $v$, we see that $v \leq 2\epsilon \om(r_j)$ on $B_{r_j} \supseteq 4B_{r_{j+1}}$, which is used in \cref{estim1} along with \cref{shrunk} to obtain 
		\begin{align}\label{estim2}
			\sup\limits_{B_{r_{j+1}}} u \leq \om(r_j) +\inf\limits_{B_{r_j}} u - 2\epsilon \omega(r_j) + C_{\varkappa}\left(\frac{c_{\text{log}}}{\log\left(\frac{1}{\sigma}\right)}\right)^\beta\Biggl(2\epsilon\omega(r_j)+\tail(v_+;x_0,2r_{j+1})\Biggr)+\varkappa\tail(v_+;x_0,r_{j+1}).
		\end{align}
		We first estimate the Tail term as follows: 
		\begin{equation}\label{tailest}\begin{array}{rcl}
			\tail(v_+;x_0,r_{j+1})^{p-1}&\overset{\redlabel{713a}{a}}{\leq}& C (r_{j+1})^{sp}\int_{B_{r_j}\setminus B_{r_{j+1}}}\frac{v_+^{p-1}}{|x-x_0|^{N+ps}}\,dx + C \sigma^{sp} \,\tail(|u_j| + 2\epsilon \om(r_j);x_0,r_j)^{p-1}\nonumber\\
			&\overset{\redlabel{713b}{b}}{\leq}& C(\epsilon\omega(r_j))^{p-1} + C \sigma^{sp} \,\tail(u_j;x_0,r_j)^{p-1}\nonumber\\
			&\overset{\redlabel{713c}{c}}{\leq}& C\left(1+\frac{\sigma^{sp-\alpha(p-1)}}{\epsilon^{p-1}}\right)\left(\epsilon\omega(r_j)\right)^{p-1} \\
			&\overset{\redlabel{713d}{d}}{\leq}&  C\left(\epsilon\omega(r_j)\right)^{p-1},
		\end{array}\end{equation} 
	where we obtain \redref{713a}{a} by splitting the integral and using the bound $v\leq |u_j|+2\epsilon\omega(r_j)$ in $\RR^N$; to obtain \redref{713b}{b}, we note that $v\leq 2\epsilon\omega(r_j)$ in $B_{r_j}$ and $\tail(2\epsilon \om(r_j);x_0,r_j)^{p-1} \lesssim (\epsilon \om(r_j))^{p-1}$; to obtain \redref{713c}{c}, we made use of \cref{taildecay} and finally to obtain \redref{713d}{d}, we recall the choice of $\epsilon$ from \cref{vechoice}.
	
	 Noting that $\tail(v_+;x_0,2r_{j+1}) \lesssim \tail(v_+;x_0,r_{j+1})$ and substituting \cref{tailest} in \cref{estim2}, we get 
		\begin{align*}
			\sup\limits_{B_{r_{j+1}}} u \leq \om(r_j) +\inf\limits_{B_{r_j}} u - 2\epsilon \omega(r_j) + 2C_\varkappa\epsilon\omega(r_j)\left(\frac{c_{\text{log}}}{\log\left(\frac{1}{\sigma}\right)}\right)^\beta  + C \varkappa\epsilon\omega(r_j).
		\end{align*}
		First, we choose $\varkappa$ sufficiently small so that 
		\begin{align}\label{estim3.5}
			\sup\limits_{B_{r_{j+1}}} u - \inf\limits_{B_{r_{j+1}}} u\leq \om(r_j) +\underbrace{\inf\limits_{B_{r_j}} u  - \inf\limits_{B_{r_{j+1}}} u}_{\leq 0} - 2\epsilon \omega(r_j) + 2C_\varkappa\epsilon\omega(r_j)\left(\frac{c_{\text{log}}}{\log\left(\frac{1}{\sigma}\right)}\right)^\beta.
		\end{align}
		Now, choosing $\sigma$ sufficiently small in \cref{estim3.5}, there exists a universal constant $\theta \in (0,1)$ such that 
		\begin{align*}
			\osc\limits_{B_{r_{j+1}}} u \leq (1- \tht) \omega(r_j) =(1-\tht)\left(\frac{r_j}{r_{j+1}}\right)^{\alpha}\omega(r_{j+1})\leq \omega(r_{j+1}),
		\end{align*} holds once we choose $\alpha>0$ small enough to satisfy $(1-\tht)\sigma^{-\alpha}<1$.
		
		The case when first alternative from \cref{alt1} holds can be handled analogously.
	\end{proof}
	
	\subsection{Proof of H\"older regularity} 
	
	\begin{proof}[Proof of \cref{holderellipic}]
		Given \cref{oscdecay}, it is a short step to see that for any $0<\rho<\frac{r_0}{2}$, we have 
		\begin{align*}
			\osc_{B_{\rho}} u \leq C \left(\frac{\rho}{r_0}\right)^\alpha \omega(r_0).
		\end{align*}
		To see this, observe that there exists $j\in \NN$ such that $r_{j+1}<\rho<r_{j}$, so that 
		\begin{align}\label{doit5}
			\osc_{B_{\rho}} u \leq \osc_{B_{r_j}} \leq \omega(r_j) = \left(\frac{r_j}{r_0}\right)^{\alpha}\omega(r_0)= \left(\frac{r_{j+1}}{\sigma r_0}\right)^{\alpha}\omega(r_0)\leq \left(\frac{\rho}{\sigma r_0}\right)^{\alpha}\omega(r_0)\leq C\left(\frac{\rho}{r_0}\right)^{\alpha}\omega(r_0).
		\end{align}
		
		Now, let $x_0 \in \Omega$ be fixed and choose $r$ such that $B_{2r}(x_0) \subset \Omega$. Let $y \in B_r(x_0)$ be any point and denote by $d=|x_0-y|$ the distance between $x_0$ and $y$. Suppose first that $d<r/2$, then there exists $\rho<r/2$ such that $\rho=|x_0-y|$ and by \cref{doit5}, we get 
		\begin{align*}
			|u(x_0)-y|\leq \osc_{B_\rho} u \leq C \left(\frac{\rho}{r}\right)^{\alpha}\omega(r) = C \left(\frac{|x_0-y|}{r}\right)^{\alpha}\omega(r).
		\end{align*}
		
		On the other hand, if $d\geq r/2$ then 
		\begin{align*}
			\frac{|u(x_0) - u(y)|}{|x_0-y|} \leq \frac{2||u||_{L^\infty(B_r(x_0))}}{r} 
		\end{align*}
		
		From the definition of $\om(r) = 2 \tail(u,x_0,r) + 2 \|u\|_{L^{\infty}(B_{2r})}$, we get the desired regularity.  
		\[
		|u(x_0)-u(y)| \leq C \left(\frac{|x_0-y|}{r}\right)^\alpha\left(2\tail(u,x_0,r)+4||u||_{L^\infty(B_{2r}(x_0))}\right),
		\]
		for any $y \in B_r(x_0)$ and a universal $\alpha \in (0,1)$.
	\end{proof}
	
%
%


\end{document}